\documentclass[11pt,reqno]{amsart}
\usepackage{amsmath, amssymb, amsthm}
\usepackage{url}
\usepackage{hyperref}
\usepackage{tikz}
\usepackage{ytableau}
\usepackage[utf8]{inputenc}
\usepackage{xcolor}
\usepackage{tikz-cd}
\usepackage{hyperref}
\usepackage{nccmath}

\usepackage{youngtab}

\usetikzlibrary{decorations.pathreplacing}

\usetikzlibrary{decorations.pathreplacing}

\renewcommand{\(}{\left\(}
\renewcommand{\)}{\right\)}
\renewcommand{\[}{\left\[}
\renewcommand{\]}{\right\]}

\numberwithin{equation}{section}
\theoremstyle{plain}
\newtheorem{theorem}{Theorem}[section]
\newtheorem{lemma}[theorem]{Lemma}

\newtheorem{problem}[theorem]{Problem}

\newtheorem{corollary}[theorem]{Corollary}

\makeatletter
\def\proof{\@ifnextchar[{\@oproof}{\@nproof}}
\def\@oproof[#1][#2]{\trivlist\item[\hskip\labelsep\textit{#2 Proof of\
		#1.}~]\ignorespaces}
\def\@nproof{\trivlist\item[\hskip\labelsep\textit{Proof.}~]\ignorespaces}

\makeatother

\begin{document}
	
	\title{Log-convexity and the overpartition function} 
	
	\author{Gargi Mukherjee}
	\address{Institute for Algebra, Johannes Kepler University, Altenberger Straße 69, A-4040 Linz, Austria.}
	\email{gargi.mukherjee@dk-compmath.jku.at}
	
	\maketitle
	
	\begin{abstract}
	Let $\overline{p}(n)$ denote the overpartition function. In this paper, we obtain an inequality for the sequence $\Delta^{2}\log \  \sqrt[n-1]{\overline{p}(n-1)/(n-1)^{\alpha}}$ which states that
	\begin{equation*}
	\log \biggl(1+\frac{3\pi}{4n^{5/2}}-\frac{11+5\alpha}{n^{11/4}}\biggr) < \Delta^{2} \log \  \sqrt[n-1]{\overline{p}(n-1)/(n-1)^{\alpha}} < \log \biggl(1+\frac{3\pi}{4n^{5/2}}\biggr) \ \ \text{for}\ n \geq N(\alpha),
	\end{equation*}
	 where $\alpha$ is a non-negative real number, $N(\alpha)$ is a positive integer depending on $\alpha$ and $\Delta$ is the difference operator with respect to $n$. This inequality consequently implies $\log$-convexity of $\bigl\{\sqrt[n]{\overline{p}(n)/n}\bigr\}_{n \geq 19}$  and $\bigl\{\sqrt[n]{\overline{p}(n)}\bigr\}_{n \geq 4}$. Moreover, it also establishes the asymptotic growth of $\Delta^{2} \log \  \sqrt[n-1]{\overline{p}(n-1)/(n-1)^{\alpha}}$ by showing $\underset{n \rightarrow \infty}{\lim} \Delta^{2} \log \  \sqrt[n]{\overline{p}(n)/n^{\alpha}} = \dfrac{3 \pi}{4 n^{5/2}}.$
	\end{abstract}

\hspace{0.65 cm} \textbf{Mathematics Subject Classifications.} Primary 05A20; 11N37.\\
\vspace{0.3 cm}
\hspace{0.8 cm} \textbf{Keywords.} Log-convexity; Overpartitions.

\section{Introduction}\label{intro}
An overpartition of $n$ is a nonincreasing sequence of natural numbers whose sum is $n$ in which the first occurrence of a number
may be overlined and $\overline{p}(n)$ denotes the number of overpartitions of $n$. For convenience, define $\overline{p}(0)=1$. For example, there are $8$ overpartitions of $3$ enumerated by $3, \overline{3}, 2+1, \overline{2}+1, 2+\overline{1}, \overline{2}+\overline{1},1+1+1, \overline{1}+1+1$. Systematic study of overpartition began with the work of Corteel and Lovejoy \cite{CorteelLovejoy}, although it has been studied under different nomenclature that dates back to MacMahon. Analogous to Hardy-Ramanujan-Rademacher formula for partition function (cf. \cite{RamanujanHardy},\cite{Rademacher}), Zuckerman \cite{Zuckerman} gave a formula for $\overline{p}(n)$ that reads
\begin{equation}\label{Zuckerman}
\overline{p}(n)=\frac{1}{2\pi}\underset{2 \nmid k}{\sum_{k=1}^{\infty}}\sqrt{k}\underset{(h,k)=1}{\sum_{h=0}^{k-1}}\dfrac{\omega(h,k)^2}{\omega(2h,k)}e^{-\frac{2\pi i n h}{k}}\dfrac{d}{dn}\biggl(\dfrac{\sinh \frac{\pi \sqrt{n}}{k}}{\sqrt{n}}\biggr),
\end{equation}
where
\begin{equation*}
\omega(h,k)=\text{exp}\Biggl(\pi i \sum_{r=1}^{k-1}\dfrac{r}{k}\biggl(\dfrac{hr}{k}-\biggl\lfloor\dfrac{hr}{k}\biggr\rfloor-\dfrac{1}{2}\biggr)\Biggr)
\end{equation*}
for positive integers $h$ and $k$. In somewhat a similar spirit as Lehmer \cite{Lehmer} obtained an error bound for the partition function, Engel \cite{Engel} provided an error term for $\overline{p}(n)$
\begin{equation}\label{Engel1}
\overline{p}(n)=\frac{1}{2\pi}\underset{2 \nmid k}{\sum_{k=1}^{N}}\sqrt{k}\underset{(h,k)=1}{\sum_{h=0}^{k-1}}\dfrac{\omega(h,k)^2}{\omega(2h,k)}e^{-\frac{2\pi i n h}{k}}\dfrac{d}{dn}\biggl(\dfrac{\sinh \frac{\pi \sqrt{n}}{k}}{\sqrt{n}}\biggr)+R_{2}(n,N),
\end{equation}
where
\begin{equation}\label{Engel2}
\bigl|R_{2}(n,N)\bigr|< \dfrac{N^{5/2}}{\pi n^{3/2}} \sinh \biggl(\dfrac{\pi \sqrt{n}}{N}\biggr).
\end{equation}
A positive sequence $\{a_n\}_{n \geq 0}$ is called $\log$-convex if for $n \geq 1$,
\begin{equation*}
a^2_n-a_{n-1}a_{n+1}\leq 0,
\end{equation*}
and it is called $\log$-concave if for $n \geq 1$,
\begin{equation*}
a^2_n-a_{n-1}a_{n+1}\geq 0.
\end{equation*}
Engel \cite{Engel} proved that $\{\overline{p}(n)\}_{n \geq 2}$ is $\log$-concave by using the asymptotic formula \eqref{Engel1} with $N=2$ followed by \eqref{Engel2}. Prior to Engel's work on overpartitions, $\log$-concavity of partition function $p(n)$ and its associated inequalities has been studied in a broad spectrum, for example see \cite{Chentalk}, \cite{Chen2}, and \cite{DeSalvoPak}. Following the same line of studies, Liu and Zhang \cite{LiuZhang} proved a list of inequalitites for overpartition function.\\
Sun \cite{Sun} initiated the study on $\log$-convexity problems associated with $p(n)$, later settled by Chen and Zheng \cite[Theorem 1.1-1.2]{Chen3}. In a more general setting, Chen and Zheng studied $\log$-convexity of $\{\sqrt[n]{p(n)/n^{\alpha}}\}_{n \geq n(\alpha)}$ (cf. \cite[Theorem 1.3]{Chen3}). Moreover, they discovered the asymptotic growth of the sequence $\Delta^2 \log \sqrt[n]{p(n)}$ (cf. \cite[Theorem 1.4]{Chen3}).\\
The main objective of this paper is to prove all the theorems \cite[Theorem 1.1-1.4]{Chen3} but in context of overpartitions. Our goal is to obtain a much more general inequality, given in Theorem \ref{mainresult}, which at once implies \cite[Theorem 1.1-1.4]{Chen3} for $\overline{p}(n)$, presented in Corollary \ref{cor1}-\ref{cor4}. More explicitly, in Theorem \ref{mainresult}, we get a somewhat symmetric upper and lower bound of $\sqrt[n]{\overline{p}(n)/n^{\alpha}}$, as shown in \eqref{mainresulteqn}. We note that the lower bound presented in \eqref{mainresulteqn} depicts a finer inequality than merely stating $\Delta^{2} \log \sqrt[n]{\overline{p}(n)/n^{\alpha}} > 0$ which implies $\log$-convexity. In another direction, we note that \eqref{mainresulteqn} readily suggests that $\dfrac{3\pi}{4}$ is the best possible constant so as to understand the asymptotic growth of $\Delta^{2} \log \sqrt[n]{\overline{p}(n)/n^{\alpha}}$, given in Corollary \ref{cor4}.\\
For $\alpha \in \mathbb{R}_{\geq 0}$, define
$r_{\alpha}(n):=\sqrt[n]{\overline{p}(n)/n^{\alpha}}$.

\begin{theorem}\label{mainresult}
Let $\alpha \in \mathbb{R}_{\geq 0}$ and 
$$N(\alpha):=
\begin{cases}
\text{$\max\Biggl\{\Bigl[\dfrac{3490}{\alpha}\Bigr]+2,\Bigl \lceil \Bigl(\dfrac{4(11+5\alpha)}{3\pi}\Bigr)^4\Bigr \rceil,5505\Biggr\}$} &\quad\text{if $\alpha \in \mathbb{R}_{>0}$},\\
\text{$4522$} &\quad\text{if $\alpha = 0$.}\\
\end{cases}$$
Then for $n \geq N(\alpha)$,
\begin{equation}\label{mainresulteqn}
\log \biggl(1+\frac{3\pi}{4n^{5/2}}-\frac{11+5\alpha}{n^{11/4}}\biggr) < \Delta^{2} \log r_{\alpha}(n-1) < \log \biggl(1+\frac{3\pi}{4n^{5/2}}\biggr).
\end{equation}
\end{theorem}

\begin{corollary}\label{cor1}
	The sequence $\bigl\{\sqrt[n]{\overline{p}(n)/n^{\alpha}}\bigr\}_{n \geq  N(\alpha)}$ is $\log$-convex.
\end{corollary}
\begin{proof}
From \eqref{mainresulteqn}, it is immediate that $$\dfrac{r_{\alpha}(n+1)r_{\alpha}(n-1)}{r^2_{\alpha}(n)} > 1+\frac{3\pi}{4n^{5/2}}-\frac{11+5\alpha}{n^{11/4}}\ \ \text{for all}\ \ n \geq N(\alpha).$$
We finish the proof by observing that $$1+\frac{3\pi}{4n^{5/2}}-\frac{11+5\alpha}{n^{11/4}} >1\ \ \text{for all}\ \ n \geq N(\alpha).$$	
\end{proof}

\begin{corollary}\label{cor2}
	The sequences $\bigl\{\sqrt[n]{\overline{p}(n)/n}\bigr\}_{n \geq 19}$  and $\bigl\{\sqrt[n]{\overline{p}(n)}\bigr\}_{n \geq 4}$ are $\log$-convex.
\end{corollary}
\begin{proof}
In order to prove $\bigl\{\sqrt[n]{\overline{p}(n)/n}\bigr\}_{n \geq 19}$  and $\bigl\{\sqrt[n]{\overline{p}(n)}\bigr\}_{n \geq 4}$ are $\log$-convex, after corollary \ref{cor1}, it remains to check numerically for $19\leq n \leq 5504$ and $4 \leq n \leq 4521$, which is done in `Mathematica' interface. 	
\end{proof}

\begin{corollary}\label{cor3}
	For all $n \geq 2$, we have
	\begin{equation}\label{cor3eqn}
	\dfrac{\sqrt[n]{\overline{p}(n)}}{\sqrt[n+1]{\overline{p}(n+1)}}\biggl(1+\frac{3\pi}{4n^{5/2}}\biggr) > \dfrac{\sqrt[n-1]{\overline{p}(n-1)}}{\sqrt[n]{\overline{p}(n)}}.
	\end{equation}
\end{corollary}
\begin{proof}
It is an immediate implication of \eqref{mainresulteqn} as it is only left over to verify \eqref{cor3eqn} for $ 2\leq n \leq 4522$, which we did numerically in `Mathematica'.	
\end{proof}

\begin{corollary}\label{cor4}
	\begin{equation}\label{cor4eqn}
	\lim_{n \rightarrow \infty} n^{5/2} \Delta^{2} \log r_{\alpha}(n) = \frac{3 \pi}{4}.
	\end{equation}
\end{corollary}
\begin{proof}
Multiplying both side of \eqref{mainresulteqn} by $n^{5/2}$ and taking limit as $n$ tends to infinity, we get \eqref{cor4eqn}.
\end{proof}

\section{Proof of Theorem \ref{mainresult}}\label{mainsec}
In this section, we give a proof of Theorem \ref{mainresult}. First, we state the Lemma 2.1 \cite[Lemma 2.1]{Chen3} of Chen and Zheng which will be useful in the proofs of Lemmas \ref{Lemma2}-\ref{Lemma4}. These lemmas further direct to get upper bound and lower bound of $\Delta^{2} \log r_{\alpha}(n)$ respectively in Lemma \ref{Lemma5} and \ref{Lemma6}, finally results \eqref{mainresulteqn}.\\

\begin{lemma}\label{Lemma1} \cite[Lemma 2.1]{Chen3}
	Suppose $f(x)$ has a continuous second derivative for $x \in [n-1,n+1]$. Then there exists $c \in (n-1,n+1)$ such that
	\begin{equation}
	\Delta^2 f(n-1)=f(n+1)+f(n-1)-2f(n)=f''(c).
	\end{equation}
	If $f(x)$ has an increasing second derivative, then 
	\begin{equation}
	f''(n-1)<\Delta^2 f(n-1) < f''(n+1).
	\end{equation}
	Conversely, if $f(x)$ has a decreasing second derivative, then 
	\begin{equation}
	f''(n+1)< \Delta^2 f(n-1) < f''(n-1).
	\end{equation}
	\end{lemma} 
We start by laying out a brief outline of Engel's primary set up \cite{Engel} for proving $\log$-concavity of $\{\overline{p}(n)\}_{n \geq 2}$. Setting $N=3$ in \eqref{Engel1}, we express $\overline{p}(n)$ as
\begin{equation}\label{overpartspliteqn}
\overline{p}(n) = \overline{T}(n)+\overline{R}(n),
\end{equation}
where
\begin{eqnarray}\label{overpartdecompeqn1}
\overline{T}(n)&=& \dfrac{\overline{c}}{\overline{\mu}(n)^2}\biggl(1-\dfrac{1}{\overline{\mu}(n)}\biggr)e^{\overline{\mu}(n)},\\
\label{overpartdecompeqn2}
\overline{R}(n) &=& \dfrac{1}{8n}\biggl(1+\dfrac{1}{\overline{\mu}(n)}\biggr)e^{-\overline{\mu}(n)}+R_{2}(n,3)
\end{eqnarray}
with $\overline{c}=\dfrac{\pi^2}{8}$ and $\overline{\mu}(n)=\pi \sqrt{n}$. In order to estimate the upper and lower bound of $\Delta^{2} \log r_{\alpha}(n-1)$, it is necessary for us to express $\Delta^{2} \log r_{\alpha}(n-1)$ in the following form
\begin{eqnarray}\label{expressioneqn}
\Delta^{2} \log r_{\alpha}(n-1) &=& \Delta^2 \dfrac{1}{n-1} \log \overline{p}(n-1) - \alpha \ \Delta^2 \dfrac{1}{n-1} \log (n-1)\nonumber\\
&=& \Delta^2 \dfrac{1}{n-1} \log \overline{T}(n-1) + \Delta^2 \dfrac{1}{n-1} \log \biggl(1+ \dfrac{\overline{R}(n-1)}{\overline{T}(n-1)}\biggr)- \alpha \ \Delta^2 \dfrac{1}{n-1} \log (n-1).\nonumber\\ 
\end{eqnarray}
Define
\begin{equation}\label{def}
\overline{E}(n-1) = \log \biggl(1+ \dfrac{\overline{R}(n-1)}{\overline{T}(n-1)}\biggr)
\end{equation}
 and rewrite \eqref{expressioneqn} as
\begin{equation}\label{expressionfinalform}
\Delta^{2} \log r_{\alpha}(n-1) = \Delta^2 \dfrac{1}{n-1} \log \overline{T}(n-1) + \Delta^2 \dfrac{1}{n-1} \overline{E}(n-1)- \alpha \ \Delta^2 \dfrac{1}{n-1} \log (n-1)
\end{equation}
Therefore, in order to estimate $\Delta^{2} \log r_{\alpha}(n-1)$, it is sufficient to estimate each of the three factors, appearing on the right hand side of \eqref{expressionfinalform}. 
\begin{lemma}\label{Lemma2}
	Let 
	\begin{eqnarray}
	\overline{G}_{1}(n) &=& \dfrac{3\pi}{4(n+1)^{5/2}}-\dfrac{5 \log \overline{\mu}(n-1)}{(n-1)^3},\\
	\overline{G}_{2}(n) &=&\dfrac{3\pi}{4(n-1)^{5/2}}-\dfrac{3 \log \overline{\mu}(n+1)}{(n+1)^3}+\dfrac{4}{(n-1)^3}.
	\end{eqnarray}
Then for $n \geq 2$, we have
\begin{equation}\label{lem2eqn1}
\overline{G}_{1}(n) < \Delta^2 \dfrac{1}{n-1} \log \overline{T}(n-1) < \overline{G}_{2}(n).
\end{equation}
\end{lemma} 
\begin{proof}
Using the definition of $\overline{T}(n)$ \eqref{overpartdecompeqn1}, we write
\begin{equation}\label{lem2eqn2}
\Delta^2 \dfrac{1}{n-1} \log \overline{T}(n-1) = \sum_{i=1}^{4} \Delta^2 \ \overline{g}_{i}(n-1),
\end{equation}	
where
\begin{eqnarray}
\overline{g}_1(n) &=& \dfrac{\overline{\mu}(n)}{n},\nonumber\\
\overline{g}_2(n) &=& -\dfrac{3 \log \ \overline{\mu}(n)}{n},\nonumber\\
\overline{g}_3(n) &=& \dfrac{\log \ (\overline{\mu}(n)-1)}{n},\nonumber\\
\text{and} \ \ \overline{g}_4(n) &=& \dfrac{\log \overline{c}}{n}.\nonumber
\end{eqnarray}
It can be easily checked that for $n \geq 3$, $\overline{g}^{'''}_1(n)<0$, $\overline{g}^{'''}_2(n)>0$, $\overline{g}^{'''}_3(n)<0$, and $\overline{g}^{'''}_4(n)<0$. As a consequence, for $n \geq 3$, $\overline{g}^{''}_1(n), \overline{g}^{''}_3(n)$, and $\overline{g}^{''}_4(n)$ are decreasing, whereas $\overline{g}^{''}_2(n)$ is increasing. Applying Lemma \ref{Lemma1}, we get for $i \in \{1,3,4\}$,
\begin{equation}\label{lem2eqn3}
\overline{g}^{''}_i(n+1) < \Delta^2 \ \overline{g}_{i}(n-1) < \overline{g}^{''}_i(n-1)
\end{equation}
and 
\begin{equation}\label{lem2eqn4}
\overline{g}^{''}_2(n-1) < \Delta^2 \ \overline{g}_{2}(n-1) < \overline{g}^{''}_2(n+1).
\end{equation}
From \eqref{lem2eqn2} and \eqref{lem2eqn3}-\eqref{lem2eqn4}, we obtain for all $n \geq 3$,
\begin{equation}\label{lem2eqn5}
\Delta^2 \dfrac{1}{n-1} \log \overline{T}(n-1) < \overline{g}^{''}_1(n-1)+\overline{g}^{''}_2(n+1)+\overline{g}^{''}_3(n-1)+\overline{g}^{''}_4(n-1)
\end{equation}
and 
\begin{equation}\label{lem2eqn6}
\Delta^2 \dfrac{1}{n-1} \log \overline{T}(n-1) > \overline{g}^{''}_1(n+1)+\overline{g}^{''}_2(n-1)+\overline{g}^{''}_3(n+1)+\overline{g}^{''}_4(n+1),
\end{equation}
where 
\begin{eqnarray}\label{lem2eqn7}
\overline{g}^{''}_1(n) & =& \dfrac{3\pi}{4n^{5/2}},\\
\label{lem2eqn8}
\overline{g}^{''}_2(n) & =& \dfrac{9}{2n^3}-\dfrac{6\log \overline{\mu}(n)}{n^3},\\
\label{lem2eqn9}
\overline{g}^{''}_3(n) & =& \dfrac{2 \log (\overline{\mu}(n)-1)}{n^3}-\dfrac{5\pi}{4n^{5/2}(\overline{\mu}(n)-1)}-\dfrac{\pi^2}{4n^2 (\overline{\mu}(n)-1)^2},\\
\label{lem2eqn10}
\text{and}\ \ \overline{g}^{''}_4(n) & =& \dfrac{2 \log \overline{c}}{n^3}.
\end{eqnarray}
We first estimate the upper bound of $\Delta^2 \dfrac{1}{n-1} \log \overline{T}(n-1)$ by \eqref{lem2eqn5} and \eqref{lem2eqn7}-\eqref{lem2eqn10}. 
\begin{equation}\label{lem2eqn11}
\begin{split}
\Delta^2 \dfrac{1}{n-1} \log \overline{T}(n-1) <& \dfrac{3\pi}{4(n-1)^{5/2}}+\dfrac{9}{2(n+1)^3}-\dfrac{6\log \overline{\mu}(n+1)}{(n+1)^3}\\
&+\dfrac{2 \log (\overline{\mu}(n-1)-1)}{(n-1)^3}-\dfrac{5\pi}{4(n-1)^{5/2}(\overline{\mu}(n-1)-1)}-\dfrac{\pi^2}{4(n-1)^2 (\overline{\mu}(n-1)-1)^2}\\
& +\dfrac{2 \log \overline{c}}{(n-1)^3}\\
&=\dfrac{3\pi}{4(n-1)^{5/2}}+\overline{U}_{1}(n)+\overline{U}_{2}(n),
\end{split} 
\end{equation}
where
\begin{eqnarray}\label{lem2eqn12}
\overline{U}_{1}(n) &=& -\dfrac{6\log \overline{\mu}(n+1)}{(n+1)^3}+\dfrac{2 \log (\overline{\mu}(n-1)-1)}{(n-1)^3}\\
\label{lem2eqn13}
\text{and}\ \ \overline{U}_{2}(n) &=& \dfrac{9}{2(n+1)^3}-\dfrac{5\pi}{4(n-1)^{5/2}(\overline{\mu}(n-1)-1)}-\dfrac{\pi^2}{4(n-1)^2 (\overline{\mu}(n-1)-1)^2}+\dfrac{2 \log \overline{c}}{(n-1)^3}.\nonumber\\
\end{eqnarray}
It can be easily check that for all $n \geq 2$,
\begin{equation}\label{lem2eqn14}
\overline{U}_{2}(n) < \dfrac{4}{(n-1)^3}.
\end{equation}
For an upper bound of $\overline{U}_{1}(n)$, we observe that for all $n \geq 15$,
\begin{equation}\label{lem2eqn15}
\dfrac{2}{(n-1)^3}< \dfrac{3}{(n+1)^3} \ \ \text{and}\ \  \log (\overline{\mu}(n)-1) < \log \overline{\mu}(n+1),
\end{equation}
that is,
\begin{equation}\label{lem2eqn16}
\dfrac{2 \log (\overline{\mu}(n-1)-1)}{(n-1)^3} < \dfrac{3 \log \overline{\mu}(n+1)}{(n+1)^3}.
\end{equation}
Consequently for $n \geq 15$ we get,
\begin{equation}\label{lem2eqn17}
\overline{U}_{1}(n) < -\dfrac{3\log \overline{\mu}(n+1)}{(n+1)^3}
\end{equation}
Invoking \eqref{lem2eqn14} and \eqref{lem2eqn17} into \eqref{lem2eqn11}, we have for $n \geq 15$,
\begin{equation}\label{lem2eqn18}
\Delta^2 \dfrac{1}{n-1} \log \overline{T}(n-1) < \dfrac{3\pi}{4(n-1)^{5/2}}-\dfrac{3\log \overline{\mu}(n+1)}{(n+1)^3}+\dfrac{4}{(n-1)^3}=\overline{G}_2(n).
\end{equation}
For lower bound of $\Delta^2 \dfrac{1}{n-1} \log \overline{T}(n-1)$, using \eqref{lem2eqn6} and \eqref{lem2eqn7}-\eqref{lem2eqn10} we obtain
\begin{equation}\label{lem2eqn19}
\begin{split}
\Delta^2 \dfrac{1}{n-1} \log \overline{T}(n-1) > & \dfrac{3\pi}{4(n+1)^{5/2}}+\dfrac{9}{2(n-1)^3}-\dfrac{6\log \overline{\mu}(n-1)}{(n-1)^3}\\
&+\dfrac{2 \log (\overline{\mu}(n+1)-1)}{(n+1)^3}-\dfrac{5\pi}{4(n+1)^{5/2}(\overline{\mu}(n+1)-1)}-\dfrac{\pi^2}{4(n+1)^2 (\overline{\mu}(n+1)-1)^2}\\
& +\dfrac{2 \log \overline{c}}{(n+1)^3}\\
&=\dfrac{3\pi}{4(n+1)^{5/2}}+\overline{L}_{1}(n)+\overline{L}_{2}(n),
\end{split}
\end{equation}
where
\begin{eqnarray}\label{lem2eqn20}
\overline{L}_{1}(n) &=& -\dfrac{6\log \overline{\mu}(n-1)}{(n-1)^3}+\dfrac{2 \log (\overline{\mu}(n+1)-1)}{(n+1)^3}\\
\label{lem2eqn21}
\text{and}\ \ \overline{L}_{2}(n) &=& \dfrac{9}{2(n-1)^3}-\dfrac{5\pi}{4(n+1)^{5/2}(\overline{\mu}(n+1)-1)}-\dfrac{\pi^2}{4(n+1)^2 (\overline{\mu}(n+1)-1)^2}+\dfrac{2 \log \overline{c}}{(n+1)^3}.\nonumber\\
\end{eqnarray}
Similarly as before, one can check that for $n \geq 9$,
\begin{equation}\label{lem2eqn22}
\overline{L}_{2}(n) >0 \ \ \ \text{and}\ \ \ \overline{L}_{1}(n) > -\dfrac{5 \log \overline{\mu}(n-1)}{(n-1)^3}.
\end{equation}
\eqref{lem2eqn19} and \eqref{lem2eqn22} yield for $n \geq 9$,
\begin{equation}\label{lem2eqn24}
\Delta^2 \dfrac{1}{n-1} \log \overline{T}(n-1) > \dfrac{3\pi}{4(n+1)^{5/2}}-\dfrac{5 \log \overline{\mu}(n-1)}{(n-1)^3}= \overline{G}_{1}(n).
\end{equation}
\eqref{lem2eqn18} and \eqref{lem2eqn24} together imply \eqref{lem2eqn1} for $n \geq 15$. We finish the proof by checking \eqref{lem2eqn1} numerically for $2\leq n \leq 14$.
\end{proof}
\begin{lemma}\label{Lemma3}
	For $n \geq 38$,
	\begin{equation}\label{lem3eqn1}
	\bigl|\Delta^2 \ \dfrac{1}{n-1} \overline{E}(n-1) \bigr|<\dfrac{5}{n-1}e^{-\dfrac{\overline{\mu}(n-1)}{12}}.
	\end{equation}
\end{lemma} 
\begin{proof}
Using \eqref{def}, we get for $n \geq 2$, 
\begin{equation}\label{lem3eqn2}
\Delta^2 \ \dfrac{1}{n-1} \overline{E}(n-1) = \dfrac{1}{n+1} \log (1+\overline{e}(n+1)) -\dfrac{2}{n} \log (1+\overline{e}(n))+\dfrac{1}{n-1} \log (1+\overline{e}(n-1)),
\end{equation}
where 
\begin{equation*}
\overline{e}(n) = \dfrac{\overline{R}(n)}{\overline{T}(n)}.
\end{equation*}
Taking absolute value of $\Delta^2 \ \dfrac{1}{n-1} \overline{E}(n-1)$ in \eqref{lem3eqn2}, we obtain for all $n \geq 2$,
\begin{equation}\label{lem3eqn3}
\bigl|\Delta^2 \ \dfrac{1}{n-1} \overline{E}(n-1) \bigr| \leq  \dfrac{1}{n+1} |\log (1+\overline{e}(n+1))| +\dfrac{2}{n} |\log (1+\overline{e}(n))|+\dfrac{1}{n-1} |\log (1+\overline{e}(n-1))|.
\end{equation}
Therefore, it is enough to estimate $|\overline{e}(n)|$. Before proceed to estimate, let us recall the bound of Engel \cite{Engel}(cf. \eqref{Engel2}) for $N=3$ that yields for $n \geq 1$,
\begin{equation}\label{lem3eqn4}
|R_{2}(n,3)|< \dfrac{9 \sqrt{3}}{2 n \ \overline{\mu}(n)} e^{\overline{\mu}(n)/3}
\end{equation}
by making use of the fact that $\sinh (x) < \dfrac{e^x}{2}$ for $x>0$. Recalling the definitions in \eqref{overpartdecompeqn1}-\eqref{overpartdecompeqn2}, we obtain 
\begin{eqnarray}\label{lem3eqn5}
|\overline{e}(n)| &=& \Biggl|\dfrac{\biggl(1+\dfrac{1}{\overline{\mu}(n)}\biggr)}{\biggl(1-\dfrac{1}{\overline{\mu}(n)}\biggr)}e^{- 2\overline{\mu}(n)}+ \dfrac{R_{2}(n,3)}{\dfrac{1}{8n}\biggl(1-\dfrac{1}{\overline{\mu}(n)}\biggr)e^{\overline{\mu}(n)}}\Biggr|\nonumber\\
& \leq & \dfrac{\biggl(1+\dfrac{1}{\overline{\mu}(n)}\biggr)}{\biggl(1-\dfrac{1}{\overline{\mu}(n)}\biggr)}e^{- 2\overline{\mu}(n)} + \dfrac{36\sqrt{3}}{\overline{\mu}(n) \biggl(1-\dfrac{1}{\overline{\mu}(n)}\biggr)}e^{- 2\overline{\mu}(n)/3}\ \  \ (\text{by}\ \eqref{lem3eqn4})\nonumber\\
&=& \dfrac{e^{- 2\overline{\mu}(n)/3}}{\overline{\mu}(n)-1}\Biggl[\bigl(\overline{\mu}(n)+1\bigr)e^{- 4\overline{\mu}(n)/3}+36\sqrt{3}\Biggr]\nonumber\\
& =& \dfrac{e^{- \overline{\mu}(n)/12}}{\overline{\mu}(n)-1}\Biggl[\Bigl(\bigl(\overline{\mu}(n)+1\bigr)e^{- 4\overline{\mu}(n)/3}+36\sqrt{3}\Bigr)e^{- \overline{\mu}(n)/2}\Biggr]e^{- \overline{\mu}(n)/12}.
\end{eqnarray}
It can be easily check that 
\begin{equation}\label{lem3eqn6}
\dfrac{e^{- \overline{\mu}(n)/12}}{\overline{\mu}(n)-1} < 1 \ \ \ \text{for all}\ n \geq 1
\end{equation}
and 
\begin{equation}\label{lem3eqn7}
\Bigl(\bigl(\overline{\mu}(n)+1\bigr)e^{- 4\overline{\mu}(n)/3}+36\sqrt{3}\Bigr)e^{- \overline{\mu}(n)/2} < 1 \ \ \ \text{for all}\ n \geq 7.
\end{equation}
Invoking \eqref{lem3eqn6} and \eqref{lem3eqn7} into \eqref{lem3eqn5}, we obtain for $n \geq 7$
\begin{equation}\label{lem3eqn8}
|\overline{e}(n)| < e^{- \overline{\mu}(n)/12}
\end{equation}
and consequently for $n \geq 38$,
\begin{equation}\label{lem3eqn9}
e^{- \overline{\mu}(n)/12} < \dfrac{1}{5}.
\end{equation}
Putting together \eqref{lem3eqn8} and \eqref{lem3eqn9}, we get for all $n \geq 38$,
\begin{equation}\label{lem3eqn10}
|\overline{e}(n)|<\dfrac{1}{5}.
\end{equation}
Next we note that for all $n \geq 38$,
\begin{equation}\label{lem3eqn11}
|\log (1+ \overline{e}(n))| \leq \dfrac{|\overline{e}(n)|}{1-|\overline{e}(n)|} < \dfrac{5}{4} \ |\overline{e}(n)|
\end{equation}
because of the fact that, for $|x|<1$, $$|\log (1+x)|< \dfrac{|x|}{1-|x|}.$$
From \eqref{lem3eqn3} and \eqref{lem3eqn11}, we obtain for $n \geq 38$,
\begin{equation}\label{lem3eqn12}
\bigl|\Delta^2 \ \dfrac{1}{n-1} \overline{E}(n-1) \bigr| < \dfrac{5}{4} \Bigl(\dfrac{|\overline{e}(n+1)|}{n+1}+2 \dfrac{|\overline{e}(n)|}{n}+\dfrac{|\overline{e}(n-1)|}{n-1}\Bigr).
\end{equation}
Plugging \eqref{lem3eqn8} into \eqref{lem3eqn12}, we have for $n \geq 38$,
\begin{eqnarray}\label{lem3eqn13}
\bigl|\Delta^2 \ \dfrac{1}{n-1} \overline{E}(n-1) \bigr| &<&\dfrac{5}{4} \Bigl(\dfrac{e^{-\overline{\mu}(n+1)/12}}{n+1}+2 \dfrac{e^{-\overline{\mu}(n)/12}}{n}+\dfrac{e^{-\overline{\mu}(n-1)/12}}{n-1}\Bigr)\nonumber\\
& < & \dfrac{5}{n-1} e^{-\overline{\mu}(n-1)/12}
\end{eqnarray}
because the sequence $\Bigl\{\dfrac{1}{n}\ e^{-\overline{\mu}(n)/12}\Bigr\}_{n \geq 1}$ is decreasing. 
\end{proof}
\begin{lemma}\label{Lemma4}
	For $\alpha \in \mathbb{R}_{>0}$ and $n \geq 7$,
	\begin{equation}\label{lem4eqn1}
-\dfrac{2\alpha \log (n-1)}{(n-1)^3}+\dfrac{3\alpha}{(n-1)^3}	< -\alpha \  \Delta^2 \dfrac{1}{n-1} \log (n-1)<-\dfrac{2\alpha \log (n+1)}{(n+1)^3}+\dfrac{3\alpha}{(n+1)^3}.
	\end{equation}
\end{lemma} 
\begin{proof}
We observe that, for $n \geq 7$,
\begin{equation*}
\Biggl(-\dfrac{\log n}{n}\Biggr)^{'''}= -\dfrac{11}{n^4}+\dfrac{6 \log n}{n^4}>0.
\end{equation*}	
Setting $f(n):= -\dfrac{\log n}{n}$ and applying Lemma \ref{Lemma1}, we obtain for $n \geq 7$,
\begin{equation}\label{lem4eqn2}
-\dfrac{2\log (n-1)}{(n-1)^3}+\dfrac{3}{(n-1)^3}< -\Delta^2 \dfrac{1}{n-1} \log (n-1) < -\dfrac{2\log (n+1)}{(n+1)^3}+\dfrac{3}{(n+1)^3}.
\end{equation}
Since $\alpha$ is a positive real number, from \eqref{lem4eqn2}, we obtain \eqref{lem4eqn1}.
\end{proof}
\begin{lemma}\label{Lemma5}
	For $\alpha \in \mathbb{R}_{\geq 0}$ and $n \geq 4021$,
	\begin{equation}\label{lem5eqn1}
	\Delta^{2} \log r_{\alpha}(n-1) < \log \biggl(1+\frac{3\pi}{4n^{5/2}}\biggr).
	\end{equation}
\end{lemma}
\begin{proof}
Using \eqref{lem2eqn1}, \eqref{lem3eqn1} and \eqref{lem4eqn1} into \eqref{expressionfinalform}, we obtain for $n \geq 38$,
\begin{equation}\label{lem5eqn2}
\Delta^{2} \log r_{\alpha}(n-1) < \overline{G}_2(n)+\dfrac{5}{n-1}e^{-\dfrac{\overline{\mu}(n-1)}{12}}-\dfrac{2\alpha \log (n+1)}{(n+1)^3}+\dfrac{3\alpha}{(n+1)^3}.
\end{equation}
Note that for all $n \geq 4$,
\begin{equation}\label{lem5eqn3}
-\dfrac{2\alpha \log (n+1)}{(n+1)^3}+\dfrac{3\alpha}{(n+1)^3} \leq 0
\end{equation}
and for $n \geq 4021$,
\begin{equation}\label{lem5eqn4}
\dfrac{5}{n-1}e^{-\dfrac{\overline{\mu}(n-1)}{12}} < \dfrac{5}{(n-1)^3}.
\end{equation}
Therefore from \eqref{lem5eqn3}-\eqref{lem5eqn4}, for all $n \geq 4021$, it follows that
\begin{equation}\label{lem5eqn5}
\Delta^{2} \log r_{\alpha}(n-1) < \dfrac{3\pi}{4(n-1)^{5/2}}-\dfrac{3 \log \overline{\mu}(n+1)}{(n+1)^3}+\dfrac{9}{(n-1)^3}.
\end{equation}
Apparently, for all $n \geq 93$, 
\begin{equation}\label{lem5eqn6}
\dfrac{3\pi}{4(n-1)^{5/2}}-\dfrac{3 \log \overline{\mu}(n+1)}{(n+1)^3}+\dfrac{9}{(n-1)^3} < \dfrac{3\pi}{4n^{5/2}}-\dfrac{9\pi^2}{32n^{5}}.
\end{equation}
Using the fact that for $x >0$, $\log (1+x) > x-\dfrac{x^2}{2}$, from \eqref{lem5eqn5} and \eqref{lem5eqn6}, we finally arrive at 
\begin{equation}\label{eqn5lem7}
\Delta^{2} \log r_{\alpha}(n-1) < \log \Bigl(1+\dfrac{3\pi}{4n^{5/2}}\Bigr).
\end{equation}
\end{proof}

\begin{lemma}\label{Lemma6}
	For $\alpha >0$ and $n \geq \max \Biggl\{\Bigl[\dfrac{3490}{\alpha}\Bigr]+2,\Bigl \lceil \Bigl(\dfrac{4(11+5\alpha)}{3\pi}\Bigr)^4\Bigr \rceil,5505\Biggr\}$,
	\begin{equation}\label{lem6eqn1}
	\Delta^{2} \log r_{\alpha}(n-1) > \log \biggl(1+\frac{3\pi}{4n^{5/2}}-\frac{11+5\alpha}{n^{11/4}}\biggr).
	\end{equation}
\end{lemma}
\begin{proof}
Using \eqref{lem2eqn1}, \eqref{lem3eqn1} and \eqref{lem4eqn1} into \eqref{expressionfinalform}, we obtain for $n \geq 250$,
\begin{equation}\label{lem6eqn2}
\Delta^{2} \log r_{\alpha}(n-1) > \overline{G}_1(n)-\dfrac{5}{n-1}e^{-\dfrac{\overline{\mu}(n-1)}{12}}-\dfrac{2\alpha \log (n-1)}{(n-1)^3}+\dfrac{3\alpha}{(n-1)^3}.
\end{equation}
It is easy to check that for $n \geq \text{max}  \Bigl\{\Bigl[\dfrac{3490}{\alpha}\Bigr]+2,4522\Bigr\}:=N_1(\alpha)$,
\begin{equation}\label{lem6eqn3}
-\dfrac{5}{n-1}e^{-\dfrac{\overline{\mu}(n-1)}{12}}+\dfrac{3\alpha}{(n-1)^3} > \dfrac{3\alpha}{(n-1)^3}-\dfrac{10470}{(n-1)^4}>0.
\end{equation}
Therefore for all $n \geq  N_1(\alpha) $,
\begin{equation}\label{lem6eqn4}
\Delta^{2} \log r_{\alpha}(n-1) > \dfrac{3\pi}{4(n+1)^{5/2}}-\dfrac{5 \log \overline{\mu}(n-1)}{(n-1)^3}-\dfrac{2\alpha \log (n-1)}{(n-1)^3}.
\end{equation}
It is immediate that for $n \geq 11$,
\begin{equation}\label{lem6eqn5}
\log \overline{\mu}(n-1)< \log (n-1)
\end{equation}
and for $n \geq 5505$,
\begin{equation}\label{lem6eqn6}
\log (n-1) < (n-1)^{1/4}.
\end{equation}
Putting \eqref{lem6eqn5} and \eqref{lem6eqn6} into \eqref{lem6eqn4}, we obtain for $n \geq \text{max}\{N_1(\alpha),5505\}$,
\begin{eqnarray}\label{lem6eqn7}
\Delta^{2} \log r_{\alpha}(n-1) &>& \dfrac{3\pi}{4(n+1)^{5/2}} -\dfrac{5+2\alpha}{(n-1)^{11/4}}
\end{eqnarray}
It remains to show that 
\begin{equation}\label{lem6eqn8}
\dfrac{3\pi}{4(n+1)^{5/2}} -\dfrac{5+2\alpha}{(n-1)^{11/4}} > \dfrac{3\pi}{4n^{5/2}} -\dfrac{11+5\alpha}{n^{11/4}}.
\end{equation}
For $n \geq \ \text{max}\Biggl\{\Biggl \lceil \Biggr(\dfrac{15 \pi}{8 (\alpha+1)}\Biggr)^{4/3}\Biggr \rceil,5\Biggr \} := N_2(\alpha)$, it follows that
\begin{equation}\label{lem6eqn9}
\dfrac{11+5\alpha}{n^{11/4}}-\dfrac{5+2\alpha}{(n-1)^{11/4}}\underset{n \geq 5}{>} \dfrac{1+\alpha}{n^{11/4}}>\dfrac{15\pi}{8 n^{7/2}} \underset{n \geq 1}{>} \dfrac{3\pi}{4}\Biggl(\dfrac{1}{n^{5/2}}-\dfrac{1}{(n+1)^{5/2}}\Biggr).
\end{equation}
From \eqref{lem6eqn7} and \eqref{lem6eqn8}, we obtain for $n \geq \ \text{max}\Bigl\{N_1(\alpha), N_2(\alpha)\Bigr\}$,
\begin{equation}\label{lem6eqn10}
\Delta^{2} \log r_{\alpha}(n-1)  > \dfrac{3\pi}{4n^{5/2}} -\dfrac{11+5\alpha}{n^{11/4}}.
\end{equation}
It is easy to check that for $n \geq \Biggl \lceil \Bigl(\dfrac{4(11+5\alpha)}{3\pi}\Bigr)^4\Biggr \rceil:=N_3(\alpha)$,
\begin{equation}\label{lem6eqn11}
\dfrac{3\pi}{4n^{5/2}} -\dfrac{11+5\alpha}{n^{11/4}}>0
\end{equation}
and using the fact that for $x>0$, $x>\log (1+x)$, we finally get for $n \geq \text{max}\{N_1(\alpha),N_3(\alpha),5505\}$ $\Bigl(\text{since}, N_3(\alpha) > N_2(\alpha) \ \text{for}\  \alpha >0\Bigr)$,
\begin{equation}\label{lem6eqn12}
\Delta^{2} \log r_{\alpha}(n-1) > \log \Bigl(1+\dfrac{3\pi}{4n^{5/2}} -\dfrac{11+5\alpha}{n^{11/4}}\Bigr).
\end{equation}
\end{proof}

\emph{Proof of Theorem \ref{mainresult}}: For $\alpha \in \mathbb{R}_{>0}$, from \eqref{lem5eqn1} and \eqref{lem6eqn1} we obtain for all $n \geq \ \text{max} \Biggl\{\Bigl[\dfrac{3490}{\alpha}\Bigr]+2,\Bigl \lceil \Bigl(\dfrac{4(11+5\alpha)}{3\pi}\Bigr)^4\Bigr \rceil,5505\Biggr\}$,
\begin{equation}\label{thmeqn1}
	\log \biggl(1+\frac{3\pi}{4n^{5/2}}-\frac{11+5\alpha}{n^{11/4}}\biggr)<\Delta^{2} \log r_{\alpha}(n-1) < \log \biggl(1+\frac{3\pi}{4n^{5/2}}\biggr).
\end{equation}
For $\alpha =0$, we have already seen that for $n \geq 4021$,
\begin{equation}\label{thmeqn2}
\Delta^{2} \log r_{\alpha}(n-1) < \log \biggl(1+\frac{3\pi}{4n^{5/2}}\biggr).
\end{equation}
For $\alpha =0$, from \eqref{lem6eqn2} we get for $n \geq 250$,
\begin{equation}\label{thmeqn3}
\Delta^{2} \log r_{\alpha}(n-1) > \overline{G}_1(n)-\dfrac{5}{n-1}e^{-\dfrac{\overline{\mu}(n-1)}{12}}.
\end{equation}
Following the same approach, it can be checked that for $n \geq 4522$,
\begin{equation}\label{thmeqn4}
-\dfrac{5}{n-1}e^{-\dfrac{\overline{\mu}(n-1)}{12}} > -\dfrac{10470}{(n-1)^4}
\end{equation}
and consequently for $n \geq 476$,
\begin{equation}\label{thmeqn5}
\overline{G}_1(n)-\dfrac{10470}{(n-1)^4} > \dfrac{3\pi}{4n^{5/2}}-\dfrac{11}{n^{11/4}}>0.
\end{equation}
So, for $\alpha=0$, by \eqref{thmeqn3}-\eqref{thmeqn5}, we obtain for $n \geq 4522$,
\begin{equation}\label{thmeqn6}
\Delta^{2} \log r_{\alpha}(n-1) > \log \biggl(1+\frac{3\pi}{4n^{5/2}}-\frac{11}{n^{11/4}}\biggr).
\end{equation}
Putting \eqref{thmeqn2} and \eqref{thmeqn6}, for $n \geq 4522$, it follows that
\begin{equation}\label{thmeqn7}
\log \biggl(1+\frac{3\pi}{4n^{5/2}}-\frac{11}{n^{11/4}}\biggr)<\Delta^{2} \dfrac{1}{n-1} \log \overline{p}(n-1)< \log \biggl(1+\frac{3\pi}{4n^{5/2}}\biggr).
\end{equation}
This finishes the proof.
\qed

\section{Conclusion}\label{Conclude}
We conclude this paper by considering the following problem;
\begin{problem}
Let $\alpha$ be a non-negative real number. Then for each $r\geq 1$, does there exists a positive integer $N(r,\alpha)$ so that for all $n \geq N(r,\alpha)$, one can obtain both upper bound and lower bound of $(-1)^r \Delta^r \log r_{\alpha}(n)$ that finally shows the asymptotic growth of $(-1)^r \Delta^r \log r_{\alpha}(n)$ as $n$ tends to infinity?
\end{problem}
For $r=2$, we have already seen that one can estimate $(-1)^r \Delta^r \log r_{\alpha}(n)$, as given in Theorem \ref{mainresulteqn} and its asymptotic growth is reflected in Corollary \ref{cor4}.
\begin{center}
	\textbf{Acknowledgements}
\end{center}
The author would like to express sincere gratitude to her advisor Prof. Manuel Kauers for his valuable suggestions on the paper. The research was funded by the Austrian Science Fund (FWF): W1214-N15, project DK13.

\end{document}